\documentclass[11pt]{article}
\usepackage{booktabs}
\usepackage{caption}
\usepackage{mathrsfs}
\usepackage{amsmath}
\usepackage{amsfonts,amsthm,amssymb,mathrsfs,bbding}
\usepackage{graphics,multicol}
\usepackage{graphicx}
\usepackage{color}
\usepackage{enumerate}
\usepackage{caption}
\usepackage{rotating}
\usepackage{lscape}
\usepackage{longtable}
\allowdisplaybreaks[4]
\usepackage{tabularx}
\usepackage[colorlinks=true,anchorcolor=blue,filecolor=blue,linkcolor=blue,urlcolor=blue,citecolor=blue]{hyperref}
\usepackage{extarrows}
\usepackage{cite}
\usepackage{tikz}
\usepackage{float}
\usepackage{latexsym,bm}
\usepackage{mathtools}
\evensidemargin=\oddsidemargin\topmargin=-1.5cm

\newtheorem{thm}{Theorem}

\newtheorem{lem}[thm]{Lemma}

\newtheorem*{remark}{Remark}

\newtheorem{conj}[thm]{Conjecture}

\newtheorem{claim}{Claim}
\theoremstyle{definition}

\addtocounter{section}{0}
\begin{document}

\title{\bf On a conjecture of Kolokolnikov on algebraic connectivity}
\author{
Cheng Chi\thanks{School of Mathematical Sciences, Shanghai Jiao Tong University, 800 Dongchuan Road, Shanghai 200240, China.
Email: {\tt chengchi@sjtu.edu.cn}. Supported by National Key R\&D Program of China under grant No. 2022YFA1006400 and National Natural Science Foundation of China No. 12571376.}
\and
Junjie Wang\thanks{College of Mathematics and Statistics, Hunan Normal University, Changsha,  Hunan, 410081, China.
Email: {\tt wangjunjie2827@163.com}. Supported by Hunan Provincial Innovation Foundation for Postgraduate (No. CX20250747).}
\and
Jiaxin Zheng\thanks{School of Mathematical Sciences, East China Normal University, 500 Dongchuan Road, Shanghai 200241, China.
Email: {\tt fabulousxjz@163.com}.}
}

\date{}
\maketitle
{\flushleft\large\bf Abstract } For a graph $G$, let $\alpha(G)$ be the second smallest eigenvalue of the Laplacian matrix of $G$, also known as the algebraic connectivity. Algebraic connectivity plays an important role in characterizing the connectivity of graphs and convergence properties of networks. Kolokolnikov conjectured that among all graphs on $n$ vertices with exactly $2n-4$ edges, $\alpha(G)\leq 2$ and one of the maximizers is the complete bipartite graph whose two parts have sizes two and $n-2$, respectively. In this paper, we completely resolve this conjecture.
\begin{flushleft}
  \textbf{Keywords:} Algebraic connectivity; Laplacian eigenvalue; spectral graph theory.
\end{flushleft}

\section{Introduction}

Let $G$ be a graph with $n$ vertices. Let $A(G)$ be the adjacency matrix of $G$, and let $D(G)$ be the diagonal matrix whose diagonal entries are the vertex degrees. The Laplacian matrix of $G$ is $L(G)=D(G)-A(G)$, and let its eigenvalues be arranged in nondecreasing order as $0=\lambda_1(G)\leq \lambda_2(G)\leq \cdots\leq \lambda_n(G)$. The second smallest Laplacian eigenvalue $\lambda_2(G)$, introduced by Fiedler, is called the \emph{algebraic connectivity} of $G$ and is denoted by $\alpha(G)$. In particular, $\alpha(G)>0$ if and only if $G$ is connected \cite{F73}.

Algebraic connectivity is not only a spectral measure of graph connectivity, but also a parameter that characterizes the performance of dynamic processes on networks, such as consensus in multi-agent networks and random walks on graphs. For multi-agent networks, for instance, a larger algebraic connectivity of the graph that represents the interactions among agents leads to a faster convergence rate of a typical consensus algorithm\cite{OM04}. This connection is relevant to communication topology design for multi-robot and UAV networks. In such systems, each vehicle typically exchanges information with only a subset of the other vehicles, and
the communication graph affects the speed at which local state or estimate information propagates through the group. Algebraic
connectivity has therefore been used in graph-based models of UAV swarms and in studies of connectivity maintenance and recovery after communication disruptions \cite{WAE16,WangUAV18,HuUAV25}.
Moreover, algebraic connectivity provides a quantitative measure of network connectivity and robustness, as it is closely related to classical vertex and edge connectivities. Consequently, algebraic connectivity has attracted considerable attention in the design and optimization of communication networks and other complex network systems. In addition, larger algebraic connectivity is associated with more robust synchronization of dynamical systems and more efficient information dissemination in random-walk processes \cite{OFT17}.

In general, determining the graph that maximizes algebraic connectivity among all graphs with $n$ vertices and $m$ edges is a difficult problem. Several extremal cases are known. For example, when $m=n-1$, the star graph $K_{1,n-1}$ is the unique maximizer of algebraic connectivity among all trees on $n\geq 3$ vertices, and its algebraic connectivity is equal to $1$ \cite{GMS90}. More generally, Ogiwara, Fukami, and Takahashi characterized stars, complete bipartite graphs, and other graph families that serve as global or local maximizers among graphs with fixed order and size under certain explicit conditions \cite{OFT17}. In this paper, we focus on the algebraic connectivity of $n$-vertex graphs with at most $2n-4$ edges. All graphs considered in this paper are undirected and simple. Let $G$ be a graph with vertex set $V(G)$ and edge set $E(G)$, and let $e(G)$ denote the number of edges in $G$.

For $1\leq t\leq n-1$, let $K_{t,n-t}$ denote the complete bipartite graph whose two vertex classes have sizes $t$ and $n-t$. Kolokolnikov proposed the following conjecture.

\begin{conj}[Kolokolnikov \cite{K15}]\label{conj::kolokolnikov}
  Among all graphs on $n$ vertices with exactly $2(n-2)$ edges, the complete bipartite graph $K_{2,n-2}$ maximizes the algebraic connectivity.
\end{conj}

Kolokolnikov\cite{K15} also verified the conjecture \ref{conj::kolokolnikov} for graphs with at most $13$ vertices. The main contribution of this paper is to prove the following Theorem \ref{thm::main1}, which confirms Conjecture \ref{conj::kolokolnikov}.

\begin{thm}\label{thm::main1}
  Let $G$ be a graph on $n\geq 4$ vertices. If $e(G)\leq 2n-4$, then $\alpha(G)\leq 2$.
\end{thm}

Since $e(K_{2,n-2})=2n-4$ and $\alpha(K_{2,n-2})=2$, Theorem \ref{thm::main1} immediately implies Conjecture \ref{conj::kolokolnikov}.

\begin{remark}
  The equality case of the upper bound in Theorem \ref{thm::main1} is not unique. In particular, when $n=10$, the complete bipartite graph $K_{2,n-2}$ and the Petersen graph are two examples of graphs attaining the bound.
\end{remark}

During the revision of this paper, we became aware of the recent preprint \cite{ZLRL26}, which also proves Conjecture \ref{conj::kolokolnikov} for $n\geq 123$.

\section{Preliminaries}

\subsection{Notation and terminology}
For $v\in V(G)$, let $d_G(v)$ denote the \emph{degree} of $v$, let $N_G(v)$ denote its \emph{open neighborhood}, and let $N_G[v]=N_G(v)\cup\{v\}$ denote its \emph{closed neighborhood}. We write $\delta(G)$ and $\Delta(G)$ for the minimum degree and maximum degree of $G$, respectively. For a vertex set $X\subseteq V(G)$, define $N_G(X)=\{v\in V(G)\backslash X: N_G(v)\cap X\neq\varnothing\}$ and $N_G[X]=X\cup N_G(X)$. For $S\subseteq V(G)$ and $v\in V(G)$, let $N_S(v)=N_G(v)\cap S$ and $d_S(v)=|N_S(v)|$. More generally, set $N_S(X)=N_G(X)\cap S$ and $N_S[X]=X\cup N_S(X)$. For $S\subseteq V(G)$, the graph $G[S]$ is the subgraph induced by $S$ and $G\backslash S$ is the subgraph induced by $V(G)\backslash S$. For $X\subseteq S$, let $e_S(X)=e(G[X])$. We also use the abbreviation $e(S)=e(G[S])$.

A subgraph $H$ of $G$ is a \emph{spanning subgraph} if $V(H)=V(G)$. A \emph{spanning tree} of a connected graph $G$ is a spanning subgraph of $G$ that is a tree, that is, connected and acyclic.

We denote by $K_r$, $P_r$, and $C_r$ the complete graph, path, and cycle on $r$ vertices, respectively. The graph $K_{1,r-1}$ is called a \emph{star}, and $2K_2$ denotes the disjoint union of two edges. The \emph{cyclomatic number} of a connected graph $H$ is $\mu(H)=e(H)-|V(H)|+1$.
A set $D\subseteq V(G)$ is a \emph{dominating set} if every vertex in $V(G)\backslash D$ has a neighbor in $D$, equivalently, if $N_G[D]=V(G)$. It is a \emph{connected dominating set} if, in addition, the induced subgraph $G[D]$ is connected. The \emph{distance} between two vertices $u$ and $v$ is the length of a shortest $(u,v)$-path.

\subsection{Some results on algebraic connectivity and $3$-connected graph}
We will need the following two results about algebraic connectivity.

\begin{lem}[Fiedler \cite{F73}]\label{lem::1}
  If $H$ is a spanning subgraph of $G$, then
  \[
    \alpha(H)\leq \alpha(G).
  \]
\end{lem}

\begin{lem}[Cvetkovi\'{c}, Rowlinson, and Simi\'{c} \cite{DPS10}]\label{lem::2}
  For any graph $G$ and any $U\subseteq V(G)$,
  \[
    \alpha(G)\leq \alpha(G-U)+|U|.
  \]
\end{lem}

To prove Theorem \ref{thm::main1}, we first establish the following structural theorem, which plays a crucial role in the proof of Theorem \ref{thm::main1}. For two disjoint subsets $S,T\subseteq V(G)$, let $E_G(S,T)$ be the set of edges with one endpoint in $S$ and the other in $T$, and let $e_G(S,T)=|E_G(S,T)|$. For $A\subseteq V(G)$, define $\partial_G(A):=e_G(A,V(G)\backslash A)$.
A nonempty set $A\subseteq V(G)$ is called a \emph{good set} if $\partial_G(A)\leq 2|A|$. Two nonempty disjoint good sets $X$ and $Y$ form a \emph{good pair} if $e_G(X,Y)=0$.

A graph with at least four vertices is \emph{$3$-connected} if deleting any set of at most two vertices leaves a connected graph; equivalently, it has no vertex cut of size at most two.

\begin{lem}\label{thm::main2}
  Let $G$ be a $3$-connected graph on $n\geq 4$ vertices, and suppose that $e(G)=2n-4$. Then $G$ contains a good pair.
\end{lem}

\section{Proof of Theorem \ref{thm::main1}}

In this section, we prove Theorem \ref{thm::main1}, assuming Lemma \ref{thm::main2} holds.

\begin{proof}[Proof of Theorem~\ref{thm::main1}. (Assuming Lemma \ref{thm::main2})]
  Let $G$ satisfy the hypotheses of Theorem \ref{thm::main1}. If $G$ is disconnected, then $\alpha(G)=0$, and there is nothing to prove. Suppose next that $G$ has a vertex cut $U$ with $1\leq |U|\leq 2$. Since $G-U$ is disconnected, Lemma \ref{lem::2} gives $\alpha(G)\leq \alpha(G-U)+|U|=|U|\leq 2$. If $e(G)<2n-4$, add edges until a graph $G'$ with $e(G')=2n-4$ is obtained. Adding edges preserves $3$-connectivity. Since $G$ is a spanning subgraph of $G'$, Lemma \ref{lem::1} yields $\alpha(G)\leq \alpha(G')$. Thus we can assume that $G$ is a $3$-connected graph with exactly $2n-4$ edges.

  By Lemma \ref{thm::main2}, $G$ admits a good pair $(X,Y)$. Let $\mathbf{x}\in\mathbb{R}^n$ be an $n$-dimensional real vector. The coordinate of \(\mathbf x\) at a vertex $v$ will be denoted by \(x_v\). The vector $\mathbf{x}$ is defined entrywise as
  $$
    x_v=
    \begin{cases}
      \frac{1}{|X|},  & v\in X,                        \\
      -\frac{1}{|Y|}, & v\in Y,                        \\
      0,              & v\in V(G)\backslash (X\cup Y).
    \end{cases}
  $$
  Then $\mathbf x \perp \mathbf e$, where $\mathbf e$ is the all-one vector. By the Rayleigh quotient characterization, we obtain
  $$
    \begin{aligned}
      \alpha (G) & =\inf_{\mathbf{x}\in\mathbb{R}^n\backslash{\{\mathbf 0\}}, \mathbf x \perp \mathbf e} \frac{\mathbf x^{\mathsf T}L(G)\mathbf x}{\mathbf x^{T} \mathbf x} \\
                 & \leq \frac{\mathbf x^{\mathsf T}L(G)\mathbf x}{\mathbf x^{T} \mathbf x}                                                                                  \\
                 & =\frac{\sum_{uv \in E(G)}(x_{u}-x_{v})^2}{\sum_{u \in V(G)}x_{u}^{2}}                                                                                    \\
                 & \le \frac{(\frac{1}{|X|})^2 \cdot 2|X|+(\frac{1}{|Y|})^2 \cdot 2|Y|}{(\frac{1}{|X|})^2 \cdot |X|+(\frac{1}{|Y|})^2 \cdot |Y|}                            \\
                 & =2.
    \end{aligned}
  $$
  This completes the proof.
\end{proof}

\section{Proof of Lemma \ref{thm::main2}}

The  elementary lemma will be used to deal with the small order cases in the proof of Lemma \ref{thm::main2}.

\begin{lem}\label{lem::5}
  If $H$ is connected, $\Delta(H)\leq 3$, and $H$ has no induced copy of $2K_2$, then $|V(H)|\leq 7$.
\end{lem}

\begin{proof}
  If $v_1v_2v_3v_4v_5$ is an induced path, then the subgraph induced by $\{v_1,v_2,v_4,v_5\}$ would be $2K_2$. Thus $H$ has no induced $P_5$.

  First suppose that $H$ contains an induced $5$-cycle $C=v_1v_2v_3v_4v_5v_1$. We claim that $V(C)$ is a dominating set. Otherwise, choose a shortest path from a vertex of $V(H) \backslash V(C)$ to $V(C)$ such that the distance from that vertex to $V(C)$ is at least $2$, and let $xy$ be the last edge of this path before reaching $V(C)$, where $y$ has a neighbor in $C$ and $x$ has none. Since $y$ is adjacent to $x$ and $\Delta(H)\le3$, it has at most two neighbors on $V(C)$. Furthermore, there exists an edge $uv$ in $C$ such that  neither of its endpoints is adjacent to $y$. Then the two edges $xy$ and $uv$ induce a copy of $2K_2$, a contradiction.

  Hence every vertex outside $V(C)$ has a neighbor on $V(C)$. If such a vertex $x$ has a neighbor in $V(C)$, we may assume without loss of generality that $x$ is adjacent to $v_1$, then the vertex $x$ must also be adjacent to $v_{3}$ or $v_{4}$, otherwise the edges $xv_1$ and $v_{3}v_{4}$ would induce a $2K_2$. Thus every vertex outside $V(C)$ is incident to at least two edges with one endpoint in $V(C)$. Each vertex of $C$ already has two neighbors on the cycle and has degree at most three, so $\partial_{H}(C) \le 5$. There are therefore at most two vertices outside $C$, and $|V(H)|\le7$.

  Now suppose that $H$ has no induced $C_5$. Choose a minimal connected dominating set $D$. We claim that $H[D]$ contains no induced $3$-path $P=abc$. We proceed by contradiction. Assume that $H[D]$ contains an induced $3$-path $P$. By the minimality of $D$, the set $D\backslash\{a\}$ is not a connected dominating set. If $D\backslash\{a\}$ is not a dominating set, then there is a vertex $a'\notin D$ with $N_H(a')\cap D=\{a\}$. In particular, $a'$ is adjacent to $a$ and to neither $b$ nor $c$. If instead $D\backslash\{a\}$ still dominates $H$ but $H[D\backslash\{a\}]$ is disconnected, then $b$ and $c$ lie in the same component because $bc\in E(H)$. Choose another component $Q$ in $D\backslash\{a\}$. Since $H[D]$ is connected, $a$ has a neighbor $a'\in Q$. This vertex $a'$ is adjacent to neither $b$ nor $c$, since either adjacency would place $a'$ in the component containing $b$ and $c$. Thus in both cases there is a vertex $a'$ adjacent to $a$ and to neither $b$ nor $c$.

  Applying the same argument to $c$ gives a vertex $c'$ adjacent to $c$ and to neither $a$ nor $b$. Clearly, the vertices $a',a,b,c,c'$ are pairwise distinct. Among them, the only possible edge in addition to $a'a$, $ab$, $bc$, $cc'$ is $a'c'$. Hence they induce a $P_5$ when $a'c'\notin E(H)$ and an induced $C_5$ when $a'c'\in E(H)$. Both cases lead to a contradiction.

  It follows that $H[D]$ is connected and has no induced $P_3$, and therefore $H[D]$ is a complete graph. Let $|D|=d$. Since $\Delta(H)\le3$, we have $d\le4$. Each vertex of $D$ has $d-1$ neighbors in $D$ and therefore at most $4-d$ neighbors outside $D$. Since $D$ dominates $H$, $|V(H)|\le d+d(4-d)=d(5-d)\le6$.
  Combining the two cases proves the lemma.
\end{proof}

We now introduce the notation used throughout the proof of Lemma \ref{thm::main2}. Assume that $G$ is a $3$-connected graph with $n$ vertices and $2n-4$ edges. Let $S=\{v\in V(G):d_G(v)=3\}$ and $T=V(G)\backslash S$. Set $|S|=s$ and $|T|=t$. Define $w(x)=4-d_T(x)$ and for any $P\subseteq T$, define $w(P)=\sum_{x\in P}w(x)$.

If $S$ is an independent set, then $e(T)=e(G)-e(S,T)=2n-4-3s=2t-s-4$ follows from the definition of $S$ and $s+t=n$. Furthermore, we have
\begin{equation}\label{eq::1}
  w(T)=\sum_{x\in T}w(x)=4t-2e(T)=4t-2(2t-s-4)=2s+8.
\end{equation}

Let $h$ denote the number of connected components of $G[T]$.

\begin{lem}\label{lem::3}
  There exist pairwise disjoint nonempty connected sets $P_1,\ldots,P_k\subseteq T$ such that $2\leq w(P_i)\leq 4(1\le i\le k)$. More precisely, for each connected component $H$ of $G[T]$, there exists a collection of pairwise disjoint connected sets together with at most one connected residual set $Q_H \subseteq H$. The union of all connected sets and all nonempty residual sets over the components of $G[T]$ is a partition of $T$. Moreover, every nonempty residual set $Q_H$ satisfies $w(Q_H)\leq 1$. Furthermore, if $S$ is an independent set, then
  \begin{equation}
    k\geq\left\lceil\frac{2s+8-h}{3}\right\rceil.
    \nonumber
  \end{equation}
\end{lem}

\begin{proof}
  For each connected component of $G[T]$, fix an arbitrary spanning tree and choose
  a root vertex. Then we apply a postorder traversal of the rooted spanning tree, from the leaves toward the root.

  We first describe the procedure of the postorder traversal. The procedure starts from the leaves. When a vertex $v$ is processed, all subtrees attached to $v$ away from the root vertex have already been decomposed. Each such subtree has been partitioned into connected sets and possibly one remaining connected set containing the vertex adjacent to $v$ in that subtree. Denote these remaining sets by $Q_u$, where $u$ ranges over all neighbors of $v$ in the rooted spanning tree that are farther from the root. We then combine $v$ with all nonempty remaining sets $Q_u$ and define $$R_v=\{v\}\cup\bigcup_{u}Q_u.$$

  The set $R_v$ is connected. Indeed, each nonempty $Q_u$ contains the vertex $u$, and the tree edge $uv$ connects $u$ with $v$. Hence every such remaining set is attached to $v$, and therefore their union with $v$ induces a connected subgraph.

  If $w(R_v)\ge2$, extract $R_v$ as a new connected set and let $Q_v=\varnothing$. Otherwise $w(R_v)\le 1$, and we put $Q_v = R_v$. Repeating this procedure until the root vertex is reached produces, for each connected component $H$ of $G[T]$, a collection of pairwise disjoint connected sets and at most one remaining connected residual set $Q_H$.

  We next prove that the resulting decomposition satisfies $2\leq w(P_i)\leq 4(1\le i\le k)$ and $k\geq\left\lceil\frac{2s+8-h}{3}\right\rceil$. If $v$ is not the root vertex of its spanning tree, then it has at most $d_T(v)-1$ neighbors farther from the root vertex, we have $$w(R_v)=w(v)+\sum_uw(Q_u)\leq(4-d_T(v))+(d_T(v)-1)=3.$$
  If $v$ is the root vertex of its spanning tree, then it has at most $d_T(v)$ such neighbors, and hence $$w(R_v)\leq(4-d_T(v))+d_T(v)=4.$$
  This proves $2\leq w(P_i)\leq 4(1\le i\le k)$.

  Fix an arbitrary connected component $H$, and let $a$ denote the number of connected sets produced by this component. If the root vertex belongs to the residual set $Q_H$, then every connected set $P_{i}$ satisfies $w(P_{i})\leq  3$ and the residual set $Q_H$ satisfies $w(Q_{H})\le1$, so $w(H) \leq 3a+1$. If the root vertex belongs to some connected set $P_i$, then $w(P_{i})\leq 4$,  whereas every remaining connected set $P_j$ with $j\neq i$ satisfies $w(P_{j})\leq 3$, and there exists no residual set. So $w(H) \leq 3(a-1)+4=3a+1.$

  Since $S$ is an independent set, summing over all $h$ connected components and by \eqref{eq::1}, we obtain $2s+8=w(T)\leq 3k+h$,
  which is equivalent to $k\geq\left\lceil\frac{2s+8-h}{3}\right\rceil$.
\end{proof}

We now prove Lemma \ref{thm::main2}.

\begin{proof}[Proof of Lemma \ref{thm::main2}]
  Suppose for a contradiction that there exists a counterexample $G$, which is a $3$-connected graph on $n$ vertices with $2n-4$ edges.

  Since $G$ is $3$-connected, we have $\delta(G)\ge3$. For $n\leq7$, $e(G)\ge\frac{3n}{2}>2n-4$, so no $3$-connected graph satisfies the hypotheses.

  Suppose that $n=8$. Then $e(G)=2n-4=12$, so $G$ is cubic. By Lemma \ref{lem::5}, $G$ contains an induced $2K_2$. Let $X$ and $Y$ be the endpoint sets of its two edges. Then $e_G(X,Y)=0$, and
  $\partial_G(X)=\partial_G(Y)=3+3-2=4$.
  Thus $(X,Y)$ is a good pair, a contradiction.

  Suppose that $n=9$. Then $e(G)=2n-4=14$. Under the condition $\delta(G)\ge3$, the degree sequence must be $4,3,3,3,3,3,3,3,3$. Recall that $S=\{v\in V(G):d_G(v)=3\}$. Each vertex of $S$ has at most one neighbor outside $S$, namely the unique vertex of degree four, and hence $\delta(G[S])\ge2$. Since $G$ is 3-connected, the induced subgraph $G[S]$ is connected. By Lemma \ref{lem::5}, $G[S]$ contains an induced $2K_2$. Similar to the case with $n=8$, we can get a good pair, again a contradiction.

  Assume from now on that $n\geq 10$.

  \begin{claim}\label{claim::1}
    Let $A\subseteq V(G)$ be nonempty vertex set and $G[A]$ is connected. Assume that $\partial_G(A)\leq2|A|-\sigma$ and $\sigma\geq0$, then $6|A|\ge n+3+2\sigma$.
  \end{claim}

  \begin{proof}
    Let $B=N_G(A)$ and $C=V(G)\backslash N_G[A]$, where $|A|=a$ and $|B|=b$. Let $C_1,\cdots,C_k$ be the connected components of $G[C]$, and let $c_i=|C_i|$. Recall $\mu(C_i)=e(G[C_i])-c_i+1\geq 0$ be the cyclomatic number of the connected graph $G[C_i]$. There are no edges between $A$ and any $C_i$. Since $A$ is a good set, if some $C_i$ is also a good set, then there exists a good pair, which yields a contradiction. Hence, we have
    \begin{equation}\label{eq::5}
      \partial_G(C_i)\geq2c_i+1
    \end{equation}
    for $1\leq i\leq k$.

    For a vertex set $U$, define $\chi(U)=\sum_{v\in U}(d_G(v)-4)$. Since $e(G)=2n-4$, we have
    \begin{equation}\label{eq::6}
      \chi(V(G))=\sum_{v\in V(G)}(d_G(v)-4)=2e(G)-4n=-8.
    \end{equation}
    By inequality \eqref{eq::5},
    \begin{equation*}
      \begin{aligned}
        \chi(C_i)
         & =\sum_{v\in V(C_i)}(d_G(v)-4)           \\
         & =2e(G[C_i])+\partial_G(C_i)-4c_i        \\
         & =2(c_i-1+\mu(C_i))+\partial_G(C_i)-4c_i \\
         & =\partial_G(C_i)-2c_i-2+2\mu(C_i)       \\
         & \geq-1.
      \end{aligned}
    \end{equation*}
    Substituting these inequalities into \eqref{eq::6} gives
    $
      -8= \chi(V(G))=\chi(A)+\chi(B)+\sum_{i=1}^k\chi(C_i)
      \geq \chi(A)+\chi(B)-k,
    $
    and therefore
    \begin{equation}\label{eq::7}
      k\ge\chi(A)+\chi(B)+8.
    \end{equation}

    Using inequality \eqref{eq::5}, we have
    \begin{equation}\label{eq::8}
      e_G(B,C)=\sum_{i=1}^k\partial_G(C_i)\geq2|C|+k.
    \end{equation}
    On the other hand, summing degrees over $B$, we obtain
    $4b+\chi(B)=\sum_{v\in B}d_G(v)=2e_G(B)+\partial_G(A)+e_G(B,C)$.
    Consequently,
    \begin{equation}\label{eq::9}
      e_G(B,C)\leq 4b+\chi(B)-\partial_G(A).
    \end{equation}
    Combining \eqref{eq::7}--\eqref{eq::9} and using $|C|=n-a-b$ gives
    \begin{equation}\label{eq::10}
      6b\geq 2n-2a+\partial_G(A)+\chi(A)+8.
    \end{equation}
    Every vertex of $B$ has at least one neighbor in $A$, so $b\le \partial_G(A)$. Hence \eqref{eq::10} implies
    $
      6\partial_G(A)\geq 2n-2a+\partial_G(A)+\chi(A)+8$.
    Since $\chi(A)=2e_G(A)+\partial_G(A)-4a$,
    we obtain
    \begin{equation}\label{eq::11}
      4\partial_G(A)+6a\geq 2n+2e_G(A)+8.
    \end{equation}
    The graph $G[A]$ is connected, so $e_G(A)\geq a-1$. Substituting this into \eqref{eq::11} gives $4\partial_G(A)+4a\geq 2n+6$. Since $\partial_G(A)\leq2a-\sigma$, and therefore $4(2a-\sigma)+4a\ge2n+6$. Rearranging proves this claim.
  \end{proof}

  \begin{claim}\label{claim::2}
    The set $S$ is an independent set.
  \end{claim}

  \begin{proof}
    Suppose that $u,v\in S$ and $uv\in E(G)$. Then $\partial_G(\{u,v\})=3+3-2=4=2|\{u,v\}|$. Claim \ref{claim::1}, with $\sigma=0$, gives $12=6|\{u,v\}|\geq n+3$,
    so $n\leq 9$, a contradiction.
  \end{proof}

  \begin{claim}\label{claim::3}
    If $x\in T$ and $d_T(x)\leq 2$, then there is a connected good set of size $d_G(x)-1$.
  \end{claim}

  \begin{proof}
    Since $d_T(x)\leq 2$, we have $d_S(x)=d_G(x)-d_T(x)\geq d_G(x)-2$. Choose any set $J\subseteq N_S(x)$ of size $d_G(x)-2$, and let $J_x=\{x\}\cup J$. By Claim \ref{claim::2}, $G[J_x]$ is a star. Then
    \begin{equation*}
      \begin{aligned}
        \partial_G(J_x)
         & =d_G(x)+3(d_G(x)-2)-2(d_G(x)-2) \\
         & =2d_G(x)-2
        =2|J_x|.
      \end{aligned}
    \end{equation*}
    Thus $J_x$ is a connected good set.
  \end{proof}

  \begin{claim}\label{claim::4}
    $n \geq 32$.
  \end{claim}

  \begin{proof}
    By Claim \ref{claim::2} and $e(G)=2n-4$, we have $e(T)=e(G)-e(S,T)=2n-4-3s\geq 0$. Hence $3s\leq 2n-4$. Since $\sum_{v\in V(G)}d_G(v)=2e(G)=4n-8$, we have $s\geq 8$, and consequently $n\geq 14$.

    First suppose that $n=14$. Then necessarily $s=8$, $t=n-s=6$, and $e(T)=2n-4-3s=0$. $\sum_{v\in V(G)}d_G(v)=2e(G)=4n-8$ gives $d_G(x)=4$ for every $x\in T$, so every vertex $x\in T$ has four neighbors in $S$. For distinct $x,y\in T$, let $c(x,y)=|N_S(x)\cap N_S(y)|$.
    Suppose that $c(x,y)\leq 2$. Then $|N_S(x)\backslash N_S(y)|\geq2$, and $|N_S(y)\backslash N_S(x)|\geq 2$.
    Choose two-element vertex sets $I_x\subseteq N_S(x)\backslash N_S(y)$ and $I_y\subseteq N_S(y)\backslash N_S(x)$,
    and let $J_x=\{x\}\cup I_x$ and $J_y=\{y\}\cup I_y$. The vertex sets $I_x$ and $I_y$ are disjoint, $S$ is independent, and $T$ has no edges. Moreover, no vertex of $I_x$ is adjacent to $y$, and no vertex of $I_y$ is adjacent to $x$. Hence $e_G(J_x,J_y)=0$. Each of $J_x,J_y$ is a star with three vertices, and $\partial_G(J_x)=\partial_G(J_y)=4+2\cdot3-2\cdot2=6$.
    Thus $(J_x,J_y)$ is a good pair, a contradiction. Therefore every one of the $\binom{6}{2}=15$ pairs $\{x,y\}\subseteq T$ must satisfy $c(x,y)\geq 3$. On the other hand, each $z\in S$ has three neighbors in $T$ and contributes to exactly three unordered pairs, so $\sum_{\{x,y\}\subseteq T}c(x,y)=3s=24$, whereas the preceding lower bound is at least $15\times 3=45$, a contradiction.

    Now suppose that $n=15$. Then necessarily $s=8$, $t=n-s=7$, and $e(T)=2n-4-3s=2$, and again $d_G(x)=4$ for every $x\in T$. Let $xy\notin E(T)$ and suppose that $c(x,y)\leq 2-\max\{d_T(x),d_T(y)\}$. Then $|N_S(x)\backslash N_S(y)|=d_S(x)-c(x,y)=d_G(x)-d_T(x)-c(x,y)\geq 2$, and the same inequality holds with $x$ and $y$ interchanged. Choose two-element vertex sets $I_x\subseteq N_S(x)\backslash N_S(y)$, $I_y\subseteq N_S(y)\backslash N_S(x)$ and let $J_x=\{x\}\cup I_x$ and $J_y=\{y\}\cup I_y$. Exactly as in the preceding paragraph, the vertex sets $I_x$ and $I_y$ are disjoint, $S$ is independent, $xy$ is a nonedge of $T$. Moreover, no vertex of $I_x$ is adjacent to $y$, and no vertex of $I_y$ is adjacent to $x$. Each of $J_x,J_y$ is a star with three vertices, and $\partial_G(J_x)=\partial_G(J_y)=4+2\cdot3-2\cdot2=6$. so $(J_x,J_y)$ is a good pair. Consequently every nonedge $xy$ of $T$ must satisfy
    \begin{equation}\label{eq::12}
      c(x,y)\geq 3-\max\{d_T(x),d_T(y)\}.
    \end{equation}
    The two edges of $T$ have only two possible configurations. If they are disjoint, the degree sequence of $T$ is $(1,1,1,1,0,0,0)$, and \eqref{eq::12} requires a total common-neighbor count of at least $3\binom{3}{2}+2(3\cdot4)+2\left(\binom{4}{2}-2\right)=41$
    over all nonedges. If the two edges are adjacent, the degree sequence is $(2,1,1,0,0,0,0)$, and the corresponding lower bound is $3\binom{4}{2}+2(4\cdot2)+4+2=40$.
    This partial sum cannot exceed the sum over all unordered pairs of vertices in $T$, which is $3s=24$. Since both lower bounds exceed $24$, we obtain a contradiction.

    Finally, suppose that $ n\geq 16$. Let $L=\{x\in T:d_T(x)\leq 2\}$  and $|L|=\ell$. If some $x\in L$ satisfies $d_G(x)=4$, Claim \ref{claim::3} yields a connected good set of size three. Claim \ref{claim::1} would then imply $18\geq n+3+2\sigma(\sigma\geq 0)$, which yields $n\leq 15$, a contradiction. Thus $d_G(x)\geq 5$ for every $x\in L$, Claim \ref{claim::2}  and $n=s+t$ gives
    \begin{equation}\label{eq::13}
      \ell \leq \sum_{x\in T}(d_G(x)-4)=4n-8-3s-4t= s-8.
    \end{equation}
    On the other hand, $\sum_{x\in T}d_T(x)=2e(T)=2(e(G)-e(S,T))=2(2n-4-3s)=4t-2s-8$.
    Every vertex $x$ outside $L$ has $d_T(x)\geq 3$, so $4t-2s-8\ge3(t-\ell)$, or equivalently
    \begin{equation}\label{eq::14}
      3\ell\geq 2s+8-t.
    \end{equation}
    Combining \eqref{eq::13} and \eqref{eq::14} gives $t \geq 32-s$, and hence $n=s+t\ge32$.
  \end{proof}

  \begin{claim}\label{claim::5}
    For any $P\subseteq T$, if $G[P]$ is connected, then $\partial_G(N_S[P])-2|N_S[P]|\leq 2-w(P)$. In particular, if $w(P)\geq 2$, then $N_S[P]$ is a connected good set and every counterexample $G$ must satisfy $6|N_S[P]|\geq n-1+2w(P)$.
  \end{claim}

  \begin{proof}
    Let $|P|=p$, $|N_S(P)|=q$, and $e_G(P,S)=r$. Then $\partial_G(N_S[P])=\partial_{G[T]}(P)+3q-r$.
    It follows that
    \begin{equation}\label{eq::15}
      \partial_G(N_S[P])-2|N_S[P]|=\partial_{G[T]}(P)-2p+q-r.
    \end{equation}
    Moreover, $\partial_{G[T]}(P)=\sum_{x\in P}(4-w(x))-2e_T(P)=4p-w(P)-2e_T(P)$.
    Since $G[P]$ is connected, $e_T(P)\geq p-1$, and since every vertex of $N_S(P)$ is incident to at least one edge to $P$, we have $r\geq q$. By \eqref{eq::15}, we have
    \begin{equation*}
      \begin{aligned}
        \partial_G(N_S[P])-2|N_S[P]| & =\partial_{G[T]}(P)-2p+q-r           \\
                                     & =4p-w(P)-2e_T(P)-2p+q-r              \\
                                     & \leq 2-w(P) ~~\mbox{(as $r\geq q$)}.
      \end{aligned}
    \end{equation*}

    Note that $N_S[P]$ is connected. If $w(P)\geq 2$, then we apply Claim \ref{claim::1} with $\sigma=w(P)-2$ to obtain $6|N_S[P]|\geq n+3+2(w(P)-2)=n-1+2w(P)$.
  \end{proof}

  Define $U_0=\{x\in T:d_G(x)=4,\ d_T(x)=3\}$ and $|U_0|=u_0$.

  \begin{claim}\label{claim::6}
    $u_0\geq 24$ and $e(U_0)=0$.  Consequently, $t\geq \frac{s}{2}+38$.
  \end{claim}

  \begin{proof}
    Suppose that $x,y\in U_0$ and $xy\in E(T)$, and take $P=\{x,y\}$. Then $w(P)=2$ and $d_S(x)=d_S(y)=1$, so $|N_S[P]|\leq 4$. Claim \ref{claim::5} gives $n+3\leq 6|N_S[P]|\leq 24$, contrary to Claim \ref{claim::4}. Thus $e(U_0)=0$.

    For $x\in T\backslash U_0$. If $d_T(x)\leq 2$, Claim \ref{claim::3} gives a connected good set of size $d_G(x)-1$. By Claims \ref{claim::1} and \ref{claim::4}, we have $6(d_G(x)-1)\geq n+3+2\sigma\geq35(\sigma\geq 0)$, and hence $d_G(x)\geq 7$. Next, we show that for $x\in T\backslash U_0$, we have $w(x)\leq 2d_G(x)-8$. If $d_T(x)\leq2$, then $w(x)\leq 4\leq 2d_G(x)-8$; if $d_T(x)=3$ and $d_G(x)>4$, then $w(x)=1\leq2d_G(x)-8$; and if $d_T(x)\geq4$, then $w(x)\leq0\leq2d_G(x)-8$. Since $e(G)=2n-4$ and $n=s+t$, we have $\sum_{x\in T}d_G(x)=4n-8-3s=4t+s-8$. Therefore \eqref{eq::1} yields $$2s+8=\sum_{x\in T}w(x)= \sum_{x\in U_0}w(x)+\sum_{x\in T\backslash U_0}w(x)\leq u_0+\sum_{x\in T}(2d_G(x)-8)\leq u_0+2s-16,$$ so $u_0\geq 24$.

    Let $R=T\backslash U_0$. Each vertex $x$ in \(U_0\) satisfies $d_{T}(x)=3$. Since \(U_0\) is an independent set, all such edges connect to $R$. Thus $e_{G[T]}(U_0,R)=3u_0$. On the other hand,  by Claim \ref{claim::2} and $n=s+t$, we have $\sum_{x\in R}d_{T}(x)=2e(T)-3u_0=2(e(G)-e(S,T))-3u_0=2(2n-4-3s)-3u_0=2(2t-s-4)-3u_0$. The left-hand side is at least \(3u_0\), the number of edges from $R$ to \(U_0\). Hence $4t-2s-8\ge6u_0\ge144$, which yields $t\ge \frac{s}{2}+38$.
  \end{proof}

  Fix a decomposition $P_1,P_2,\cdots,P_k$ provided by Lemma \ref{lem::3}, and let $k$ be the number of connected sets contained in this decomposition.

  \begin{claim}\label{claim::7}
    Every counterexample $G$ satisfies $(k-6)t\le(14-k)s+k-16$.
  \end{claim}

  \begin{proof}
    We fix the decomposition above. Let $R$ be the union of all final residual sets, and let $|R|=r$. Every nonempty residual set $Q$ satisfies $w(Q) \le 1$, and the number of nonempty residual sets is at most their total number $r$ of vertices. Hence
    \begin{equation}\label{eq::20}
      w(R)\le r.
    \end{equation}
    The connected sets contain exactly $t-r$ vertices, and
    \begin{equation}\label{eq::21}
      \sum_{i=1}^k w(P_i)=w(T)-w(R).
    \end{equation}

    By Claim \ref{claim::5} and Lemma \ref{lem::3}, $N_S[P_i]$ is a connected good set satisfying $\partial_G(N_S[P_i])\le2|N_S[P_i]|-(w(P_i)-2)$, and
    \begin{equation}\label{eq::22}
      6|N_S[P_i]|\geq n-1+2w(P_i).
    \end{equation}
    The connected sets $P_i$ are pairwise disjoint. Moreover, each vertex in $S$ has only three neighbors, so there are at most three distinct indices $i$ such that it belongs to $N_S[P_i]$. Consequently,
    \begin{equation}\label{eq::23}
      \sum_{i=1}^k|N_S[P_i]|\leq t-r+3s.
    \end{equation}
    Summing \eqref{eq::22} and using \eqref{eq::21} and~\eqref{eq::23}, we obtain $6(t-r+3s) \geq k(n-1)+2(w(T)-w(R))$.
    Rearranging and applying \eqref{eq::20},
    \begin{equation*}
      \begin{aligned}
        k(n-1)+2w(T) & \leq 6(t+3s)+2w(R)-6r                    \\
                     & \leq 6(t+3s) ~~\mbox{(as $w(R)\leq r$)}.
      \end{aligned}
    \end{equation*}

    Substituting $n=s+t$ and \eqref{eq::1} and simplifying gives $(k-6)t\le(14-k)s+k-16$.
  \end{proof}

  Let the connected components of $T$ be $T_1,\dots,T_h$, and let $t_i=|T_i|$. For a vertex $x\in S$, define $a_i(x)=|N_{T_i}(x)|$, and let $\rho_i=\sum_{x\in S}a_i(x)$, $q_i=|\{x\in S:a_i(x)>0\}|$.

  \begin{claim}\label{claim::8}
    Every counterexample  $G$ satisfies $h(h-1)\leq 2s$.
  \end{claim}
  \begin{proof}
    For any two distinct components $T_i$ and $T_j$ with $i\neq j$, let $S_i^{(j)}=\{x\in S:a_i(x)>0,\ a_j(x)=0\}$, $X_i^{(j)}=T_i\cup S_i^{(j)}$
    and define $X_j^{(i)}$ similarly.  By Claim \ref{claim::2}, $X_i^{(j)}$ and $X_j^{(i)}$ are nonempty and disjoint. Note that $e(X_i^{(j)},X_j^{(i)})=0$

    Define $\Psi_i(j) :=\partial_G(X_i^{(j)})-2|X_i^{(j)}|$.  Note that $\partial_G(T_i)=\rho_i$.  Starting from $T_i$, adding a vertex $x\in S_i^{(j)}$ changes the boundary minus twice the cardinality by $1-2a_i(x)$. Hence
    $$
      \Psi_i(j)=\rho_i-2t_i+\sum_{x\in S_i^{(j)}}(1-2a_i(x)).
    $$
    Let $\overline{S}_{i}^{(j)}=\{x\in S:a_i(x)>0,\ a_j(x)>0\}$. Note that $\rho_i=\sum_{x\in S_i^{(j)}}a_i(x)+\sum_{x\in \overline{S}_{i}^{(j)}}a_i(x)$ and $q_i=|S_i^{(j)}|+|\overline{S}_{i}^{(j)}|$, then
    \begin{equation*}
      \begin{aligned}
        \sum_{x\in S_i^{(j)}}(1-2a_i(x)) & = |S_i^{(j)}|-2\sum_{x\in S_i^{(j)}}a_i(x)                                      \\
                                         & =|S_i^{(j)}|-2(\rho_i-\sum_{x\in \overline{S}_{i}^{(j)}}a_i(x))                 \\
                                         & =q_i-|\overline{S}_{i}^{(j)}|-2\rho_i+2\sum_{x\in \overline{S}_{i}^{(j)}}a_i(x) \\
                                         & =q_i-2\rho_i+\sum_{x\in \overline{S}_{i}^{(j)}}(2a_i(x)-1).
      \end{aligned}
    \end{equation*}
    Thus,
    \begin{equation}\label{eq::30}
      \begin{aligned}
        \Psi_i(j) & =\rho_i-2t_i+\sum_{x\in S_i^{(j)}}(1-2a_i(x))                          \\
                  & =\rho_i-2t_i+q_i-2\rho_i+\sum_{x\in \overline{S}_{i}^{(j)}}(2a_i(x)-1) \\
                  & =-2t_i+q_i-\rho_i
        +\sum_{x\in \overline{S}_{i}^{(j)}}
        (2a_i(x)-1).
      \end{aligned}
    \end{equation}
    If both $\Psi_i(j)\le0$ and $\Psi_j(i)\le0$, then $X_i^{(j)}$ and $X_j^{(i)}$ form a good pair, a contradiction. Therefore, since $\Psi_i(j)$ is an integer, for distinct indices $i,j$, at least one of $\Psi_i(j)$ and $\Psi_j(i)$ is positive. We now construct a directed complete graph $K_h$ such that each vertex of $K_h$ corresponds to a connected component of $T$. We orient the edge $ij$ toward vertex $i$ satisfying $\Psi_i(j)>0$. If both vertices have positive value, either orientation may be selected. By \eqref{eq::30}, $\rho_i\geq q_i$ and $t_i\geq 1$,  every edge $ij$ oriented toward $i$ satisfies
    \begin{equation}\label{eq::31}
      \sum_{x\in \overline{S}_{i}^{(j)}}(2a_i(x)-1) \ge 2t_i-q_i+\rho_i+1 \ge 3.
    \end{equation}

    We now show that each $x\in S$ contributes at most $3$ to the sum of all left-hand sides in~\eqref{eq::31}. Fix $x\in S$ and consider how its three neighbors are distributed among the components of $G[T]$.
    \begin{itemize}
      \item If all three neighbors lie in one component, then $x$ contributes to no pair of components.
      \item If the distribution is $2+1$ across two components, there is only one relevant component pair. If its edge is oriented toward the component of multiplicity two, the contribution is $2\cdot2-1=3$; if it is oriented toward the other component, the contribution is $1$. Thus the total contribution is at most $3$.
      \item If the distribution is $1+1+1$ across three components, there are three component pairs. For each pair, whichever orientation is chosen, the contribution is $1$. Thus the total contribution is $3$.
    \end{itemize}
    Summing over all elements in $S$, the total sum of these left-hand sides is at most $3s$. By inequality \eqref{eq::31}, each edge of $K_h$ contributes at least $3$. Hence $3\binom h2\le3s$, which is equivalent to $h(h-1)\le2s$.
  \end{proof}

  We next show that $k\geq 7$. Since $\sum_{v\in V(G)}d_G(v)=2e(G)=4n-8$, we have $s\geq 8$. For $s=8,9,10$, Claim \ref{claim::8} gives $h\leq 4,4,5$, respectively, and  Lemma \ref{lem::3} then gives $k\geq 7,8,8$, respectively. If $s\geq 11$, then Claim \ref{claim::8} implies $h\leq s$. Indeed, $h\geq s+1$ would give $h(h-1)\geq s(s+1)>2s$, a contradiction. Therefore $2s+8-h\geq s+8\geq 19$ and  Lemma \ref{lem::3} again gives $k\geq 7$. Thus $k-6>0$ in every case. By Claims \ref{claim::6} and \ref{claim::7}, every counterexample $G$ must satisfy
  \begin{equation}\label{eq::40}
    \Delta(s,k):= \frac{3k-34}{2}s+37k-212\leq 0.
  \end{equation}
  Since the coefficient of $k$ in $\Delta(s,k)$ is $\frac{3s}{2}+37>0$, the function $\Delta(s,k)$ is strictly increasing in $k$.

  We distinguish three cases according to the value of $s$, where $s\geq 8$. First consider $s\geq 18$. From Claim \ref{claim::8} we obtain $h\leq \frac{1+\sqrt{1+8s}}{2} \leq \sqrt{2s}+1$. Moreover, $(s-12)^2-2s=(s-8)(s-18)\geq 0$, so $\sqrt{2s}\leq s-12$ and hence $h\leq s-11$. Therefore $2s+8-h\geq s+19\geq 37$ and Lemma \ref{lem::3} gives $k\geq 13$. Since $\Delta(s,k)$ is strictly increasing in $k$, we have $\Delta(s,k)\geq\Delta(s,13)=\frac{5}{2}s+269>0$, contradicting \eqref{eq::40}.

  Suppose next that $9\leq s\leq 17$. Applying Claim \ref{claim::8}, inequalities Lemma \ref{lem::3}  and \eqref{eq::40} yields the following result.
  \[
    \begin{array}{c|c|c|c}
      \toprule
      s  & \text{upper bound for }h           & \text{lower bound for }k
         & \text{lower bound for }\Delta(s,k)                                    \\
      \midrule
      9  & 4                                  & 8                        & 39    \\
      10 & 5                                  & 8                        & 34    \\
      11 & 5                                  & 9                        & 82.5  \\
      12 & 5                                  & 9                        & 79    \\
      13 & 5                                  & 10                       & 132   \\
      14 & 5                                  & 11                       & 188   \\
      15 & 6                                  & 11                       & 187.5 \\
      16 & 6                                  & 12                       & 248   \\
      17 & 6                                  & 12                       & 249   \\
      \bottomrule
    \end{array}
  \]
  Every value is positive, again contradicting \eqref{eq::40}.

  It remains only to consider $s=8$. $\sum_{v\in V(G)}d_G(v)=2e(G)=4n-8$ gives $d_G(x)=4$ for every $x\in T$. By Claims \ref{claim::3} and \ref{claim::4}, no vertex in $T$ can satisfy $d_T(x)\leq 2$, since that would produce a connected good set of size three and Claim \ref{claim::1} would imply $n\leq 15$. On the other hand, since every vertex of $T$ has degree $4$ in $G$, its degree in $T$ satisfies $d_T(x)=3$ or $d_T(x)=4$. By \eqref{eq::1}, we have $w(T)=2s+8=24$, exactly twenty-four vertices in $T$ have $d_T(x)=3$. Thus $u_0=24$.
  By Claim \ref{claim::6}, these vertices form an independent set in $T$. Let $R=T\backslash U_0$. Every vertex $v\in R$ has $d_T(v)=4$ and $d_S(v)=0$, while $e_{G[T]}(U_0,R)=3u_0=72$. Suppose that some $v\in R$ is adjacent to two distinct vertices $x,y\in U_0$. For $P=\{x,v,y\}$, we have $w(P)=1+0+1=2$. Here $d_S(x)=d_S(y)=1$ and $d_S(v)=0$, so $|N_S(P)|\leq d_S(x)+d_S(y)+d_S(v)=2$.
  Consequently, we have $|N_S[P]|=|P|+|N_S(P)|\leq 3+2=5$.
  Claim \ref{claim::5} gives $30\geq n-1+2w(P)$, contrary to Claim \ref{claim::4}. Thus every vertex of $R$ is incident with at most one edge from $U_0$, and consequently $|R|\geq 72$. It follows further that $t=u_0+|R|\geq96$. However, Claim \ref{claim::8} gives $h\le4$ when $s=8$, and Lemma \ref{lem::3} then gives $k\geq 7$. By Claim \ref{claim::7}, we have $(k-6)t\leq(14-k)8+k-16=96-7k\leq47$. Since $k-6\geq1$, this implies $t\leq 47$, contradicting $t\geq 96$.

  This completes the proof.
\end{proof}

\section*{Declaration for the use of AI}
The Lemma \ref{lem::3} and its proof were originally suggested by AI. The authors carefully rechecked and refined the suggested version, which led to the current version. The authors of this paper bear full responsibility for the correctness of the proofs and the rigor of the exposition.

\end{document}